\documentclass[a4paper,11pt]{smfart}
\usepackage[francais]{babel}
\usepackage{amssymb}
\usepackage{amsmath}
\usepackage[T1]{fontenc}
\usepackage[francais]{babel}
\usepackage{amsmath}
\usepackage{amsthm}
\usepackage{amsfonts}
\usepackage{amssymb}
\usepackage{enumerate}
\newtheorem{theo}{Th\'eor\`eme}[]

\newtheorem{lem}[theo]{Lemme}
\newtheorem{coro}[theo]{Corollaire} 
\newtheorem{algo}[theo]{Algorithme}

\theoremstyle{definition}

\theoremstyle{remark} 
\newtheorem{rem}[theo]{Remarque}
\theoremstyle{definition}  

\newcommand{\Q}{\overline{\mathbb Q}}

\newcommand{\lambd}{{\boldsymbol{\lambda}}}
\newcommand{\f}{{\bf f}}
\renewcommand{\O}{\mathcal{O}}

\renewcommand{\k}{{\bf{k}}}

\newcommand{\bmu}{{\boldsymbol{\mu}}}

\newcommand{\w}{{\bf{w}}}

\title[]{Méthode de Mahler, transcendance et relations linéaires : aspects effectifs}

\author{Boris Adamczewski} 
\address{
Univ Lyon, Universit\'e Claude Bernard Lyon 1\\
CNRS UMR 5208, Institut Camille Jordan  \\
43 blvd du 11 novembre 1918 \\
F-69622 Villeurbanne Cedex, France}
\email{Boris.Adamczewski@math.cnrs.fr}

\author{Colin Faverjon}
\email{colin.faverjon@ac-creteil.fr}

\thanks{This project has received funding from the European Research Council (ERC) under the European Union's Horizon 2020 research and innovation programme 
under the Grant Agreement No 648132. }

\begin{abstract} Cette note est consacr\'ee aux aspects algorithmiques de la m\'ethode de Mahler. 
Dans un travail récent, nous avons utilisé un résultat de Philippon  
pour montrer qu'étant donnés une fonction 
$q$-mahl\'erienne $f(z)$ appartenant à ${\bf k}\{z\}$, où ${\bf k}$ est un corps de nombres, et un nombre algébrique $\alpha$ dans le domaine d'holomorphie de $f$, le nombre $f(\alpha)$ 
est soit transcendant, soit dans ${\bf k}(\alpha)$.   
Nous décrivons ici un algorithme permettant de trancher cette alternative. 
Plus g\'en\'eralement, \'etant donn\'es  plusieurs fonctions $q$-mahl\'eriennes $f_1(z),\cdots,f_r(z)$ 
et un nombre algébrique $\alpha$ dans le domaine d'holomorphie des $f_i$, 
nous montrons comment calculer explicitement une base de l'espace vectoriel 
des relations de d\'ependance lin\'eaire sur $\Q$ entre les nombres 
$f_1(\alpha),\cdots,f_r(\alpha)$.
\end{abstract}

\begin{document}
\maketitle

\section{Introduction}

Soit $q\geq 2$ un entier. Une fonction  $f(z)\in {\Q}\{z\}$ est  dite {\it $q$-mahl\'erienne}  
s'il existe des polyn\^omes $p_{-1}(z),p_0(z),\ldots , p_n(z)\in {\mathbb Q}[z]$, non tous nuls, 
tels que 
\begin{equation} \label{eq: mahler}
 p_{-1}(z) + p_0(z)f(z)+p_1(z)f(z^q)+\cdots + p_n(z)f(z^{q^n}) \ = \ 0. 
 \end{equation}
 \'Etant donnée une telle équation, on peut toujours se ramener explicitement à une équation 
 pour laquelle $p_0(z)\not=0$ (cf.\ Remarque \ref{rem:p0}), ce que nous supposerons 
 dans toute la suite. 
 Afin d'\'etudier une fonction $q$-mahlérienne, il est souvent commode   
 de considérer un système d'\'equations fonctionnelles de la forme : 

\begin{equation}\label{eq: systeme}
\left( \begin{array}{ c }
     f_1(z) \\
     \vdots \\
     f_n(z)
  \end{array} \right) = A(z)\left( \begin{array}{ c }
     f_1(z^q) \\
     \vdots \\
     f_n(z^q)
  \end{array} \right)  \, ,
\end{equation}
o\`u  $A(z)$ est une matrice de $\rm{Gl}_n(\overline{\mathbb Q}(z))$  et les $f_i$ sont 
des fonctions de la variable $z$, analytiques au voisinage de $z=0$. Rappelons qu'une fonction 
est solution d'une équation de type \eqref{eq: mahler} si, 
et seulement si, elle est la coordonnée d'un vecteur solution d'un système de 
type  \eqref{eq: systeme}. 
La donn\'ee d'une \'equation de type \eqref{eq: mahler} 
ou d'un système de type \eqref{eq: systeme} ne permet pas en général 
de d\'eterminer une solution $f(z)$ 
ou un vecteur solution $(f_1(z),\ldots,f_n(z))$ de façon unique. Pour lever cette ambiguïté, il 
suffit, comme l'a remarqu\'e Dumas \cite{Dumas}, de se donner un nombre suffisant de coefficients du développement de Taylor à l'origine de $f$ 
ou des fonctions $f_i$.  Dans la section \ref{par:solutions}, nous reprenons la d\'emarche de \cite{Dumas}, afin de préciser brièvement ce point. Dans la suite, nous considérerons donc que chaque 
fonction $q$-mahlérienne est donn\'ee avec une équation de type \eqref{eq: mahler} ou  un syst\`eme de type  \eqref{eq: systeme}, ainsi qu'avec un nombre suffisant de coefficients de Taylor pour la d\'eterminer de manière unique. 
Rappelons également les faits classiques suivants (voir par exemple \cite{AB,Dumas,Ra}). 
\begin{itemize}

\medskip

\item[$\bullet$] Une série formelle $f(z)\in {\Q}[[z]]$ solution d'une équation de type  
\eqref{eq: mahler} est toujours convergente au voisinage de l'origine. Elle admet de plus un prolongement méromorphe sur le disque unité 
ouvert, le cercle unité formant une frontière naturelle 
dès lors que $f(z)$ n'est pas une fraction rationnelle. 

\medskip

\item[$\bullet$] Les coefficients de Taylor d'une fonction $q$-mahlérienne    
engendrent une extension finie de $\mathbb Q$.   
\end{itemize}
 
 \medskip

D'après la remarque précédente, on ne perd donc aucune généralité à 
supposer que $f(z)$ appartient à ${\bf k}\{z\}$, où ${\bf k}$ est un corps de nombres.  
 Soient $f(z) \in {\bf k}\{z\}$ une fonction $q$-mahl\'erienne et  
 $\alpha \in {\Q}$ dans le domaine d'holomorphie de $f(z)$. En nous appuyant sur un théorème 
 de Philippon \cite{PPH}, nous avons montr\'e dans \cite{AF} que 
 l'on a toujours l'alternative suivante : soit $f(\alpha)$ est transcendant, 
 soit $f(\alpha) \in {\bf k}(\alpha)$. L'objectif principal de cette note est de montrer que cette alternative 
 peut \^etre tranch\'ee de façon algorithmique.

\begin{algo}
\label{algo: transcendance}
Soient $f(z)$ une fonction $q$-mahl\'erienne donn\'ee par une \'equation de type 
\eqref{eq: mahler} ou un système de type \eqref{eq: systeme}, et $\alpha$, $0<|\alpha|<1$, 
un nombre alg\'ebrique. On peut d\'eterminer de mani\`ere algorithmique si la fonction 
$f$ est d\'efinie au point $\alpha$, et, le cas \'ech\'eant, si $f(\alpha)$ est alg\'ebrique ou 
transcendant.
\end{algo}

L'\'etude de la transcendance des valeurs d'une fonction $q$-mahl\'erienne $f(z)$ 
est un cas particulier de l'\'etude des relations de d\'ependance lin\'eaire entre les valeurs 
de plusieurs fonctions mahl\'eriennes. Nous montrerons plus  généralement le résultat suivant.

\begin{algo}
\label{algo: relationslineaires}
Soient $f_1(z),\cdots,f_r(z)$ des fonctions $q$-mahl\'eriennes, données chacune par 
une \'equation de type \eqref{eq: mahler} ou un système de type \eqref{eq: systeme}, 
et $\alpha$, $0<|\alpha|<1$, un nombre alg\'ebrique. On peut d\'eterminer de mani\`ere 
algorithmique si les fonctions $f_i(z)$ sont toutes d\'efinies au point $\alpha$, et, le cas \'ech\'eant, 
d\'eterminer une base de l'espace vectoriel
$$
\emph{Rel}_{\Q}(f_1(\alpha),\ldots,f_r(\alpha)) := \left\{(\lambda_1,\ldots,\lambda_r)\in {\Q}^r :
\sum_{i=1}^r \lambda_if_i(\alpha) =0\right\}\, .
$$ 
\end{algo}

Afin étudier la nature des valeurs d'une fonction $q$-mahlérienne $f(z)$, il est  
souvent agréable de consid\'erer son équation inhomogène minimale, c'est-à-dire l'\'equation de type 
\eqref{eq: mahler} satisfaite par $f(z)$ pour laquelle l'entier $n$ est minimal. Cette \'equation 
est bien unique (à multiplication par une constante près) si l'on impose aux polyn\^omes $p_i(z)$ 
d'être premiers entre eux. 
Une étape intermédiaire dans l'algorithme \ref{algo: transcendance} consiste à déterminer une telle équation.  

\begin{algo}
\label{algosystemeinhomogene}

Soit $f$ une fonction $q$-malherienne donn\'ee par une équation de type \eqref{eq: mahler} ou un système de type \eqref{eq: systeme}. On peut d\'eterminer explicitement 
l'\'equation inhomog\`ene minimale de $f$.

\end{algo}

Cet algorithme, dont l'intérêt est indépendant, est décrit dans la section \ref{par:relationslin}. 
Les algorithmes \ref{algo: transcendance} et \ref{algo: relationslineaires} sont d\'etaill\'es dans la 
section \ref{par:alternative}.

\section{Détermination des relations lin\'eaires entre fonctions mahl\'eriennes}
\label{par:relationslin}

D\'eterminer les relations alg\'ebriques sur $\Q(z)$ entre des fonctions $q$-mahl\'eriennes 
est a priori une tâche ardue qui relève de la théorie de Galois aux différences 
associée à l'opérateur mahlérien $z\mapsto z^q$. 
Nous remarquions dans \cite[Th\'eor\`eme 6.1]{AF} 
que l'on peut \textit{a contrario} d\'eterminer de mani\`ere effective une base de l'espace 
vectoriel des relations de d\'ependance lin\'eaires sur $\Q(z)$ entre plusieurs fonctions 
$q$-mahl\'eriennes données. 

\begin{algo}
\label{algo:indep}
Soient $f_1(z),\ldots,f_n(z)$ des fonctions $q$-mahl\'eriennes solutions d'un système de type  \eqref{eq: systeme} et \`a coefficients dans un corps de nombre $\k$. 
On peut calculer explicitement la dimension $r$ de l'espace vectoriel engendr\'e sur $\k(z)$  
par les fonctions $f_1(z),\ldots,f_n(z)$, et, pour chaque $r$-uplet de fonctions 
$f_{i_1}(z),\ldots,f_{i_r}(z)$, tester si ces dernières sont lin\'eairement ind\'ependantes sur $\k(z)$. 
Le cas échéant, on peut déterminer un syst\`eme de la forme 
\eqref{eq: systeme} contenant uniquement ces $r$ fonctions.
\end{algo}

\begin{proof}[Description de l'algorithme \ref{algo:indep}]
Notons de mani\`ere compacte $\f(z)$ le vecteur colonne form\'e des fonctions $f_1(z),\ldots,f_n(z)$. 
D\'eveloppant chaque coordonn\'ee en s\'eries enti\`eres, on note
$$
\f(z):=\sum_{i=0}^\infty \f_i z^i \,.
$$
Soient $b(z)$ le plus petit multiple commun des d\'enominateurs des coefficients de $A(z)$, $d$ le maximum des degr\'es des coefficients de la matrice $b(z)A(z)$, et $\nu$ la valuation en $z=0$ du polyn\^ome $b(z)$.  
On pose $h:=4^nd$ et 
$$
M:= \left\lceil \frac{q^{n \left(\frac{qh+d+1}{q-1} + q + 1 \right)}(h+q) + \nu - \frac{h+d}{q-1} }{q-1} \right\rceil \, \cdot
$$ 
On d\'efinit alors la matrice de Sylvester associ\'ee \`a $\f(z)$ par 
\begin{equation*}
{\mathcal S}({\bf f}) := \left(\begin{array}{cccccc} \f_0 & \f_1 & \cdots & \f_h 
& \cdots & \f_M \\ 0 & \f_0 & \ddots & \ddots & \ddots & \f_{M-1} \\ 
\vdots & \ddots & &&& \vdots \\ 0 & \cdots & 0 & \f_0 & \cdots & \f_{M-h} \end{array}\right) \, ,
\end{equation*} 
o\`u $\f_i$ d\'esigne le vecteur colonne de $\k^n$ dont les coordonn\'ees correspondent aux 
$i$-\`emes coefficients des fonctions 
$f_1(z),\ldots,f_n(z)$. On pose 
$$
{\rm ker} {\mathcal S}(\f) := 
\left\{ { \lambd} \in ({\bf k}^n)^{h+1} \mid { \lambd} {\mathcal S}(\f) = 0\right\} \, .
$$
Si $\w=(\w_0,\ldots,\w_h)$ appartient à ${\rm ker} {\mathcal S}({\bf f})$, 
on d\'efinit le vecteur de polynômes 
$$
\w(z):=\sum_{i=0}^h \w_iz^i \,.
$$
Nous avons montr\'e, dans \cite[th\'eor\`eme 6.1]{AF}, qu'alors 
$$
\langle \w(z) , \f(z) \rangle = 0 \,,
$$ 
o\`u $\langle\, ,  \rangle$ d\'esigne le produit scalaire usuel, et que toutes les relations de 
d\'ependance lin\'eaires entre les fonctions $f_i(z)$, $1\leq i \leq n$, s'obtiennent de 
cette mani\`ere l\`a. On peut calculer de mani\`ere explicite la dimension $s$ du noyau  
${\rm ker}{\mathcal S}({\bf f})$, ainsi qu'une base $\w_1(z),\ldots,\w_{s}(z)$ de l'espace vectoriel 
$$
{\rm Rel}_{\bf k(z)}(f_1(z),\ldots,f_n(z)) :=  \left\{ (w_1(z),\ldots,w_n(z)) \in {\bf k}(z)^n : 
 \sum_{i=1}^nw_i(z)f_i(z) = 0 \right\} 
 $$ 
en d\'eterminant une base de  ${\rm ker}{\mathcal S}({\bf f})$. 
La dimension du ${\bf k}(z)$-espace vectoriel engendré par les fonctions 
$f_1(z),\ldots,f_n(z)$ vaut donc $r=n-s$, et la matrice
$$
\Lambda(z) := 
\left(\begin{array}{c} 
\w_1(z) \\
\vdots \\
\w_{s}(z)
\end{array} \right)
$$
est de rang $s$. L'\'etude des mineurs non nuls de cette matrice nous permet de d\'eterminer les parties $\{f_{i_1}(z),\ldots,f_{i_{r}}(z)\}$ form\'ees de fonctions lin\'eairement ind\'ependantes. 
En effet, fixons $i_1,\ldots,i_r$, $1\leq i_1 <\cdots< i_{r} \leq n$, des entiers distincts. 
Notons $J := \{1,\ldots,n\} \setminus \{i_1,\ldots,i_{r}\}$ 
et $\Delta_J$ le mineur de $\Lambda$ associ\'e \`a l'ensemble $J$, 
on montre qu'on a l'\'equivalence 
\begin{equation}\label{eq:equivmineurdependance}
\Delta_J=0 \Leftrightarrow f_{i,1}(z),\ldots,f_{i_{r}}(z) \text{ sont lin\'eairement d\'ependantes sur } {\bf k}(z).
\end{equation}
Supposons que $\Delta_J = 0$. Il existe un vecteur $\bmu(z) \in {\bf k}[z]^s$ non nul tel que $\boldsymbol{\kappa}(z):=\bmu(z)\Lambda(z)$ est nul sur $J$, c'est-à-dire que 
$\boldsymbol{\kappa}(z):=(\kappa_1(z),\ldots,\kappa_n(z))\in {\bf k}[z]^n$ avec $\kappa_i(z)=0$ 
pour tout $i$ dans $J$. D'autre part, la matrice $\Lambda$ étant 
de rang $s$, le vecteur $\boldsymbol{\kappa}(z)$ 
est nécessairement non nul. 
Il vient donc : 
$$
\sum_{l=1}^{r} \kappa_{i_l}(z)f_{i_l}(z)
 = \boldsymbol{\kappa}(z) \left(\begin{array}{c} f_1(z) \\ \vdots \\ f_n(z) \end{array} \right) 
 = \bmu(z) \Lambda(z)\left(\begin{array}{c} f_1(z) \\ \vdots \\ f_n(z) \end{array} \right)
 = 0 \, ,
$$
ce qui montre que les fonctions $f_{i_1}(z),\ldots,f_{i_{r}}(z)$ sont lin\'eairement d\'ependantes 
sur ${\bf k}(z)$.

R\'eciproquement, supposons que les fonctions $f_{i_1}(z),\ldots,f_{i_{r}}(z)$ sont 
lin\'eairement d\'ependantes. Il existe alors un vecteur non nul 
$\boldsymbol{\kappa}(z):=(\kappa_1(z),\ldots,\kappa_n(z))\in {\bf k}[z]^n$ tel que 
$\sum_{i=1}^n\kappa_i(z)f_i(z)=0$ 
et $\kappa_i(z)=0$ pour tout $i$ dans $J$. 
Comme les vecteurs $\w_1(z),\ldots,\w_{s}(z)$ forment une base 
de ${\rm Rel}_{\bf k(z)}(f_1(z),\ldots,f_n(z))$, il existe un vecteur non nul $\bmu(z) \in {\bf k}[z]^s$ 
tel que $\bmu(z)\Lambda(z)=\boldsymbol{\kappa}(z)$. Ainsi $\Delta_J = 0$. 

On peut donc tester algorithmiquement, pour tout choix d'entiers $i_1,\ldots,i_r$, si les $r$ fonctions $f_{i_1}(z),\ldots,f_{i_{r}}(z)$ sont lin\'eairement ind\'ependantes. 
Supposons à présent que ce soit le cas. 
Quitte \`a renum\'eroter, on peut supposer sans perte de généralité qu'il s'agit des fonctions 
$f_1(z),\ldots,f_{r}(z)$. Consid\'erons alors la matrice 
$$
S(z) := 
\left(\begin{array}{ccccccc} 
1      & 0      & \cdots & \cdots & \cdots & \cdots & 0 \\
0      & 1      & \ddots &        &        &        & \vdots \\
\vdots & \ddots & \ddots & \ddots &        &        & \vdots \\
0      & \cdots & 0      & 1      & 0      & \cdots & 0 \\
\hline
       &        &        &        &        &        & \\
       &        &        & \Lambda(z) &    &        & \\
	   &			&		&		&		 &		  &
\end{array}\right).
$$
D'apr\`es \eqref{eq:equivmineurdependance} cette matrice est inversible, et on a 
$$
S(z) \left(\begin{array}{c} f_1(z) \\ \vdots \\ f_n(z) \end{array} \right)
=\left(\begin{array}{c} f_1(z) \\ \vdots \\ f_{r}(z) \\ 0 \\ \vdots \\ 0 \end{array} \right).
$$
Notons $B(z)$ le bloc principal $r\times r$ de la matrice $S(z) A(z)  S(z^q)^{-1}$. 
On obtient par construction que 
\begin{equation}
 \label{eq:malherreduite}
 \left(\begin{array}{c} f_1(z) \\ \vdots \\ f_{r}(z) \end{array} \right) 
 = B(z)  \left(\begin{array}{c} f_1(z^q) \\ \vdots \\ f_{r}(z^q) \end{array} \right).
\end{equation}
D'autre part, les fonctions $f_1(z),\ldots,f_{r}(z)$ \'etant lin\'eairement ind\'ependantes, la matrice $B(z)$ est inversible, et le syst\`eme \eqref{eq:malherreduite} est bien un système de type \eqref{eq: systeme}.
\end{proof}

\subsection{Détermination de l'équation inhomogène minimale d'une fonction mahlérienne}

Soit $f(z)$ une fonction $q$-mahl\'erienne d\'efinie par une \'equation de type \eqref{eq: mahler}. L'espace vectoriel 
$$
V_{f}:={\rm Vect}_{\Q (z)}\{1,f(z),f(z^q),f(z^{q^2}),\cdots \}\,.
$$
est donc de dimension finie. Le lemme suivant montre que la dimension de $V_f$ 
est déterminée par l'ordre de l'\'equation inhomog\`ene minimale de $f$.

\begin{lem}
\label{lemme: dimension}
Soit $f(z)$ une fonction $q$-malherienne d'\'equation inhomog\`ene minimale
\begin{equation}
\label{eq: inhomogeneminimale}
 p_{-1}(z) + p_0(z)f(z)+p_1(z)f(z^q)+\cdots + p_n(z)f(z^{q^n}) \ = \ 0\,.
\end{equation}
La dimension de l'espace $V_{f}$ est exactement $n+1$.
\end{lem}

\begin{proof}
Notons $(V_m)_{m\geq 1}$ la suite d'espaces embo\^it\'es 
$$
V_m := {\rm Vect}_{\Q (z)}\{1,f(z),f(z^q),f(z^{q^2}),\ldots, f(z^{q^m})\}\,.
$$
Nous montrons que si $V_{m_0+1} = V_{m_0}$ pour un $m_0$ fix\'e, alors $V_m = V_{m_0}$ 
pour tout $m \geq m_0$. Il suffit en fait de montrer que dans un tel cas, 
$V_{m_0+2} = V_{m_0+1}$. 

Soit $m_0$ un entier tel que $V_{m_0+1} = V_{m_0}$. On peut trouver des fractions rationnelles $q_{-1}(z),\ldots q_{m_0}(z)$, telles que 
$$
f(z^{m_0+1}) = q_{-1}(z)+q_0(z)f(z)+ \cdots + q_{m_0}(z)f(z^{q^{m_0}})\,.
$$
En appliquant cette \'egalit\'e en $z^q$, on trouve
$$
f(z^{m_0+2}) = q_{-1}(z^q)+q_0(z^q)f(z^q)+ \cdots + q_{m_0}(z^q)f(z^{q^{m_0+1}}) \in V_{m_0+1}\,.
$$
On en d\'eduit que $V_{m_0+2} = V_{m_0+1}$.

Par minimalit\'e de l'\'equation \eqref{eq: inhomogeneminimale},  on a $V_m \subsetneq V_{m+1}$ 
pour tout $m<n-1$. D'autre part, l'\'equation \eqref{eq: inhomogeneminimale} implique l'\'egalit\'e d'espaces vectoriels, $V_{n}=V_{n-1}$. On a donc, par stationnarit\'e, $V_{f} = V_{n-1}$, et 
$\dim V_{n-1} = n+1$, ce qui termine la preuve du lemme.
\end{proof}

\begin{rem}\label{rem:p0}
Soit $f(z)$ solution d'une équation de la forme
$$
p_{-1}(z)+p_0(z)f(z) + \cdots + p_{n}(z)f(z^{q^n}) = 0\ ,
$$
où les $p_i$ sont des polynômes non tous nuls.  
Si le polynôme $p_0(z)$ est nul, on peut déterminer explicitement une équation 
$$
\tilde p_{-1}(z)+\tilde p_0(z)f(z) + \cdots + \tilde p_{n-j}(z)f(z^{q^{n-j}}) = 0\ 
$$ 
telle que $\tilde p_0(z)\not=0$, où $j$ désigne le plus petit entier positif tel que $p_j(z)\not=0$.  
En effet, en décomposant les $p_i$ selon les puissances de $z$ modulo $q^j$, 
on obtient 
$$
p_i(z) := \sum_{k=0}^{q^j-1} p_{i,k}(z^{q^j})z^k 
$$
et l'équation 
\begin{equation*}
p_{-1}(z)+ p_{j}(z)f(z^{q^j}) + \cdots + p_{n}(z)f(z^{q^n}) = 0
\end{equation*}
implique que 
$$
p_{-1,k}(z)+ p_{j,k}(z)f(z) + \cdots + p_{n,k}(z)f(z^{q^{n-j}}) = 0
$$ 
pour tout $k$ tel que $0\leq k< q^j$. 
Comme $p_j(z)\not=0$, il existe au moins un entier $k_0$ tel que $p_{j,k_0}(z)\not=0$ et il suffit de poser $\tilde p_{-1}(z)=p_{-1,k_0}(z)$ et $\tilde p_{i}(z)=p_{i+j,k_0}(z)$ pour $0\leq i\leq n-j$. 
\end{rem}

Nous sommes à présent en mesure de décrire l'algorithme \ref{algosystemeinhomogene}.

\begin{proof}[Description de l'algorithme \ref{algosystemeinhomogene}]
Nous pouvons supposer sans perte de généralité que $f$ est donnée comme coordonnée d'un 
système de type \eqref{eq: systeme}. En effet, si $f$ est donn\'ee comme solution 
d'une \'equation de type \eqref{eq: mahler}, il suffit de 
considérer le système compagnon suivant :  
$$
\left(\begin{array}{c}1 \\ f(z) \\ \vdots \\ f(z^{q^{n-1}}) \end{array} \right)
= \left(\begin{array}{ccccc}
 1 & 0 & \cdots& \cdots & 0  
 \\ - \frac{p_{-1}(z)}{p_0(z)} & \frac{p_{1}(z)}{p_0(z)} & \cdots & \cdots & \frac{p_{n}(z)}{p_0(z)}
 \\ 0 & 1 & 0 & \cdots & 0
 \\ \vdots & \ddots & \ddots & \ddots & \vdots 
 \\ 0 & \cdots & \cdots & 1 & 0
\end{array} \right)
\left(\begin{array}{c}1 \\ f(z^q) \\ \vdots \\ f(z^{q^{n}}) \end{array} \right) \,.
$$

On suppose donc que $f(z)$ est donn\'ee comme coordonnée d'un vecteur solution d'un 
syst\`eme de type \eqref{eq: systeme}. On note $(f_1(z),\ldots,f_n(z))$ un tel 
vecteur solution. Quitte \`a effectuer des permutations sur la matrice $A(z)$ et éventuellement à ajouter la fonction identiquement \'egale \`a $1$, 
en changeant la matrice $A(z)$ en
$$
\left(\begin{array}{ccc|c} & & & 0 \\ & A(z) & & \vdots \\ & & & 0 \\ \hline 
0 & \cdots & 0 & 1 \end{array}\right) \, ,
$$
on supposera désormais que $f_1(z)=f(z)$ et $f_n(z) = 1$. 
 
Le th\'eor\`eme 6.1 de \cite{AF} permet de d\'eterminer si la fonction $f(z)$ est rationnelle 
ou non, et le cas échéant, de trouver deux polynôme $A$ et $B$ premiers entre eux 
tels que $f=A/B$.  
L'équation minimale inhomogène est alors $A(z) - B(z)f(z)=0$. 

On supposera donc à présent que $f$ n'est pas une fraction rationnelle. Dans ce cas, 
$f$ est nécessairement transcendante sur ${\bf k}(z)$, de sorte que $1$ et $f$ sont linéairement indépendantes. En appliquant l'algorithme \ref{algo:indep}, on trouve une matrice $B(z)$ inversible, et des fonctions $g_1(z):=f(z),g_2(z),\ldots,g_s(z),g_{s+1}(z):=1$ lin\'eairement ind\'ependantes 
sur ${\bf k}(z)$ telles que
\begin{equation}
\label{eq:systemeintermediaire}
 \left(\begin{array}{c} g_1(z) \\ \vdots \\ g_{s+1}(z)\end{array} \right) 
 := B(z) \left(\begin{array}{c} g_1(z^q) \\ \vdots \\ g_{s+1}(z^q)\end{array} \right)\,.
\end{equation}
En inversant le syst\`eme, on obtient $f(z^q)$ en fonction de $g_1(z),\ldots,g_{s+1}(z)$,
$$
f(z^q)=g_1(z^q)=\lambd_1(z) \left(\begin{array}{c} g_1(z) \\ \vdots \\ 
g_{s+1}(z)\end{array} \right). 
$$
En it\'erant le syst\`eme \eqref{eq:systemeintermediaire} $s$ fois, et en inversant les matrices $B_l(z):=B(z)B(z^q) \cdots B(z^{q^{l-1}})$, 
on trouve, pour chaque $l \leq s$, un vecteur $\lambd_l(z)$ tel que
$$
f(z^{q^l})=g_1(z^{q^l})=\lambd_l(z)
\left(\begin{array}{c} g_1(z) \\ \vdots \\ g_{s+1}(z)\end{array} \right)\,. 
$$
On notera aussi, $\lambd_{0}:=(0,\ldots,0,1)$ de sorte que 
$$
1=\lambd_{0} \left(\begin{array}{c} g_1(z) \\ \vdots \\ g_{s+1}(z)\end{array} \right)\,.
$$
Consid\'erons alors la matrice $C(z)$, dont les lignes sont les vecteurs 
$\lambd_l(z)$, $l=0\ldots s$. On a 
$$
  \left(\begin{array}{c}1 \\ f(z^q) \\ \vdots \\ f(z^{q^s})\end{array}\right) 
= C(z)\left(\begin{array}{c} g_1(z) \\ \vdots \\ g_{s+1}(z)\end{array} \right)\,.
$$
Les fonctions $g_1(z),\ldots,g_{s+1}(z)$ \'etant lin\'eairement ind\'ependantes, 
le rang $r$ de $C(z)$ est \'egal \`a la dimension de l'espace engendr\'e sur ${\Q}(z)$ 
par les fonctions $1,f(z^q),\ldots,f(z^{q^s})$, c'est \`a dire, de l'espace $V_{g} \subset V_{f}$ où $g(z):=f(z^q)$. Ainsi $\dim V_{f}\geq r$. Montrons 
maintenant que  $\dim V_{f}=r$.  
En effet, d'apr\`es le lemme \ref{lemme: dimension}, il existe $r+1$ polyn\^omes 
$q_{-1}(z),q_1(z),\ldots,q_{r}(z)$ tels que 
$$
q_{-1}(z)+q_1(z)f(z^q) + \cdots + q_{r}(z)f(z^{q^{r}}) = 0\ .
$$
Comme dans la remarque \ref{rem:p0}, en ne regardant que les puissances de $z$ 
selon leur reste modulo $q$, on extrait  
$r$ polyn\^omes $\tilde{q}_{-1}(z),\cdots,\tilde{q}_{r-1}(z)$ tels que
\begin{equation}
\label{eq: reduction}
\tilde{q}_{-1}(z)+\tilde{q}_{0}(z)f(z) + \cdots + \tilde{q}_{r-1}(z)f(z^{q^{r-1}}) = 0\ 
\end{equation}
et $\tilde{q}_{0}(z)\not=0$. D'après le lemme \ref{lemme: dimension}, il vient  
$\dim V_{f}\leq r$ et donc $\dim V_f = r$.  
L'entier $r$ se calcule explicitement puisque 
c'est simplement le rang de la matrice $C(z)$. 

L'\'equation \eqref{eq: reduction} donne alors l'existence d'un vecteur non nul 
$
\bmu(z):=(\mu_1(z),\ldots,\mu_{r}(z),0,\ldots,0)
$ 
tel que
$$
\bmu(z)
  \left(\begin{array}{c}1 \\ f(z^q) \\ \vdots \\ f(z^{q^{s}})\end{array}\right) 
  = f(z) \ \left(= g_1(z) \right)\,.
$$
L'ind\'ependance lin\'eaire des fonctions $g_1(z),\ldots,g_{s+1}(z)$ entra\^ine que 
$$
\bmu(z)  C(z) = (1,0,\ldots,0)\,.
$$
Un tel vecteur $\bmu(z)$ se calcule donc explicitement en résolvant un syst\`eme lin\'eaire. 
On obtient alors : 
$$
\mu_1(z) + \mu_2(z) f(z^q) + \cdots + \mu_r(z) f(z^{q^{r-1}}) = f(z)\,.
$$
On note $p_0(z)$ le ppcm des dénominateurs des fractions rationnelles $\mu_i(z)$,  
puis on pose : 
$$
p_{-1}(z):= -p_0(z)\mu_1(z) \; \text{ et } \; p_i(z) := -p_0(z)\mu_{i-1}(z)\ ,\ 
\text{ pour } i\in \{1, \ldots,r\}\,.
$$
En vertu du lemme \ref{lemme: dimension}, l'\'equation
$$
p_{-1}(z) + p_0(z)f(z)+p_1(z)f(z^q)+\cdots + p_{r-1}(z)f(z^{q^{r-1}}) \ = \ 0
$$
est bien l'\'equation inhomog\`ene minimale de $f(z)$.
\end{proof}

\begin{rem}
Comme dans le cas différentiel, la notion d'\'equation minimale 
est vraiment relative à une solution donnée. 
L'exemple suivant illustre ce fait en exhibant une équation mahlérienne 
qui est minimale par rapport à l'une de ses solutions, 
tout en admettant également une solution qui est une 
fraction rationnelle.
Consid\'erons l'\'equation fonctionnelle 
\begin{equation}
\label{exemple} zf(z) -(1+2z) f(z^2) + (1+z ) f(z^4) = 0\,.
\end{equation}
On voit rapidement que la fonction identiquement \'egale \`a $1$ est solution. D'autre part, l'\'etude des relations entre les coefficients montre qu'il existe une solution analytique dont les premiers coefficients sont
$$ f(z) = z + 2z^2+z^3+3z^4+2z^5+2z^6 +z^7+4 z^8+ \cdots.$$
Nous allons voir que \eqref{exemple} est l'équation inhomogène minimale de la fonction $f(z)$ (bien qu'elle possède un terme constant nul). \'Ecrivons le système inhomogène compagnon associé à cette équation :
$$
\left(\begin{array}{c}1 \\ f(z) \\ f(z^2) \end{array} \right)
=
\left(\begin{array}{ccc}
1 & 0&0 
\\ 0 & \frac{1+2z}{z} & -\frac{1+z}{z}
\\0 &1 & 0
\end{array}\right)
\left(\begin{array}{c}1 \\ f(z^2) \\ f(z^4) \end{array} \right)\, .
$$
Comme dans la description de l'algorithme \ref{algo:indep}, le théorème \cite[Théorème 6.2]{AF} montre que les fonctions $1$, $f(z)$ et $f(z^2)$ sont linéairement indépendantes 
si et seulement si la matrice de Sylvester
$$
\mathcal S
:=\left(\begin{array}{ccccccccccc}
1&0&0&0&0&0&0&0&0&\cdots
\\
0&1&2&1&3&2&2&1&4&\cdots
\\
0&0&1&0&2&0&1&0&3&\cdots
\\
0&1&0&0&0&0&0&0&0&\cdots
\\
0&0&1&2&1&3&2&2&1&\cdots
\\
0&0&0&1&0&2&0&1&0&\cdots
\\
0&0&1&0&0&0&0&0&0&\cdots
\\
0&0&0&1&2&1&3&2&2&\cdots
\\
0&0&0&0&1&0&2&0&1&\cdots
\end{array}
\right)\, $$
est de rang maximal. Le d\'eterminant de la matrice form\'ee par les $9$ premi\`eres colonnes vaut $1$. Les fonctions
$1$, $f(z)$ et $f(z^2)$ sont donc lin\'eairement ind\'ependantes et l'\'equation \eqref{exemple} est bien l'\'equation inhomog\`ene minimale associ\'ee \`a la fonction $f(z)$. 
\end{rem}

\section{Description des algorithmes \ref{algo: transcendance} et \ref{algo: relationslineaires}}
\label{par:alternative}


L'algorithme \ref{algo: transcendance} est bien sûr un cas particulier de l'algorithme \ref{algo: relationslineaires}, mais, pour la clart\'e de l'exposition, il nous semble préférable de pr\'esenter ces deux algorithmes de mani\`ere distincte. Nous commençons par décrire l'algorithme \ref{algo: transcendance}. 

\begin{proof}[Description de l'algorithme \ref{algo: transcendance}]
Une fonction $q$-mahl\'erienne $f(z)$ étant donn\'ee par une équation de type \eqref{eq: mahler} ou un système de type \eqref{eq: systeme}, on applique d'abord l'algorithme \ref{algosystemeinhomogene} pour d\'eterminer l'\'equation  inhomog\`ene minimale associ\'ee \`a la fonction $f$. 
En \'ecrivant cette \'equation sous forme d'un système, on obtient : 
\begin{equation}
\label{eq: systememinf}
\begin{array}{rcl}
\left(\begin{array}{c}1 \\ f(z) \\ \vdots \\ \vdots \\ f(z^{q^{n-1}}) \end{array} \right)
 & =  &\left(\begin{array}{ccccc}
 1 & 0 & \cdots& \cdots & 0  
 \\ - \frac{p_{-1}(z)}{p_0(z)} & \frac{p_{1}(z)}{p_0(z)} & \cdots & \cdots & \frac{p_{n}(z)}{p_0(z)}
 \\ 0 & 1 & 0 & \cdots & 0
 \\ \vdots & \ddots & \ddots & \ddots & \vdots 
 \\ 0 & \cdots & \cdots & 1 & 0
\end{array} \right)
\left(\begin{array}{c}1 \\ f(z^q) \\ \vdots \\ \vdots \\ f(z^{q^{n}}) \end{array} \right) \\ 
&:= &A(z) \left(\begin{array}{c}1 \\ f(z^q) \\ \vdots \\ f(z^{q^{n}}) \end{array} \right)\ . \end{array}
\end{equation}
Soit $\rho$, $0<\rho<1$, un nombre réel strictement inférieur aux modules des p\^oles non nuls des coefficients de $A(z)$ et des racines non nulles du d\'eterminant de $A(z)$. 
La matrice $A(z)$ \'etant inversible et \`a coefficients dans $\Q(z)$, 
un tel $\rho$ se calcule de mani\`ere effective.

Nous montrons dans un premier temps comment déterminer si $\alpha$ est ou non un pôle de $f$.  
Observons d'abord que la fonction $f(z)$ est d\'efinie sur le disque $D(0,\rho)$. 
En effet, dans le cas contraire choisissons $\xi$ un p\^ole de module minimal pour la fonction $f$, de sorte que $|\xi|\leq \rho$. Par définition de $\rho$, le nombre $\xi$ ne serait pas un p\^ole de la matrice $A(z)$. L'\'equation \eqref{eq: systememinf} impliquerait alors que $\xi$ serait p\^ole d'une des fonctions $f(z^q),\ldots, f(z^{q^n})$, ce qui contredirait la minimalit\'e de $\xi$. 

On choisit maintenant un entier $l$ tel que $|\alpha^{q^l}| \leq \rho$. En itérant le système 
\ref{eq: systememinf}, on a : 
$$ 
\left(\begin{array}{c}1 \\ f(z) \\ \vdots \\ f(z^{q^{n-1}}) \end{array} \right)
= A_l(z)\left(\begin{array}{c}1 \\ f(z^{q^l}) \\ \vdots \\ f(z^{q^{l+n-1}}) \end{array} \right)
$$ 
o\`u $A_l(z)=A(z)\cdots A(z^{q^{l-1}})$. La minimalit\'e du syst\`eme \eqref{eq: systememinf} 
impliquent  de plus que les fonctions $1,f(z),\ldots,f(z^{q^{n-1}})$ sont lin\'eairement 
ind\'ependantes. D'apr\`es le th\'eor\`eme 1.10 de \cite{AF}, chaque p\^ole de $A_l(z)$ est 
p\^ole d'au moins une des fonctions $1,f(z),\ldots,f(z^{q^{n-1}})$. 
D'autre part, par choix de l'entier $l$, $\alpha$ n'est p\^ole d'aucune des fonctions $f(z^{q^l}),\ldots,f(z^{q^{l+n-1}})$. Par cons\'equent, $\alpha$ est un p\^ole de $f(z)$ si et seulement si 
c'est un p\^ole d'un des coefficients de la deuxième ligne de la matrice $A_l(z)$. 
Comme $A_l(z)$ se calcule explicitement, on peut 
déterminer de façon effective si la fonction 
$f$ est définie ou non au point $\alpha$. 

\medskip

 Supposons à présent que la fonction $f(z)$ est bien d\'efinie en $\alpha$. 
 
 Si $\alpha$ n'est pas une racine du d\'eterminant de $A_l(z)$, 
 alors c'est un point r\'egulier pour la matrice $A_l(z)$, par choix de l'entier $l$. 
 D'apr\`es le corollaire 1.5 de \cite{AF}, les fonctions $1,f(z),\ldots,f(z^{q^{n-1}})$ 
 \'etant lin\'eairement ind\'ependantes, le nombre $f(\alpha)$ est transcendant.

 Il reste à traiter le cas où $\alpha$ est une racine du d\'eterminant de $A_l(z)$. 
 Le nombre $f(\alpha)$ est alg\'ebrique si, et seulement si, il existe deux nombres alg\'ebriques 
$\omega_1$ et $\omega_2$, non tous nuls, tels que
$$
\omega_1 + \omega_2 f(\alpha) = 0 \,.
$$
D'apr\`es le th\'eor\`eme 1.9 de \cite{AF}, cela est \'equivalent \`a l'existence d'un vecteur non nul 
de la forme $(\omega_1,\omega_2,0,\ldots,0)$ dans le noyau à gauche de la matrice $A_l(\alpha)$. 
L'existence d'un tel vecteur peut se tester algorithmiquement.  
\end{proof}

Décrivons à présent l'algorithme \ref{algo: relationslineaires}. 

\begin{proof}[Description de l'algorithme \ref{algo: relationslineaires}]
Pour chaque fonction $f_i(z)$, $1\leq i\leq r$, on dispose d'un syst\`eme de la forme 
\eqref{eq: systeme}.
L'algorithme \ref{algosystemeinhomogene} nous permet de trouver le syst\`eme inhomog\`ene minimal associ\'e \`a chaque fonction $f_i(z)$, $1 \leq i\leq r$. Soit $\rho$ un r\'eel positif tel que les matrices des syst\`emes sont d\'efinies et inversibles sur le disque ferm\'e \'epoint\'e $D(0,\rho)^\star$, et $l$ un entier tel que $|\alpha^{q^l}| \leq \rho$. Comme dans la preuve de l'algorithme \ref{algo: transcendance}, on it\`ere $l$ fois chaque syst\`eme pour obtenir 
$$ 
\left(\begin{array}{c}1 \\ f_i(z) \\ \vdots \\ f_i(z^{q^{n_i-1}}) \end{array} \right)
= A_{i,l}(z)\left(\begin{array}{c}1 \\ f_i(z^{q^l}) \\ \vdots \\ f_i(z^{q^{l+n_i-1}}) \end{array} \right)\, 
$$  
puis déterminer si chaque fonction $f_i(z)$ est d\'efinie en $\alpha$. 

On supposera dans la suite que les fonctions $f_i(z)$ sont toutes d\'efinies en $\alpha$. 
Chaque fonction $f_i(z)$ étant également $q^l$-mahlérienne, l'algorithme \ref{algosystemeinhomogene} permet de déterminer l'équation inhomogène minimale 
associée : 
$$
p_{i,-1}(z) +p_{i,0}(z)f_i(z)+p_{i,1}(z)f_i(z^{q^l})+\cdots +f_i(z^{q^{l+n_i}}) =0\,.
$$
Ainsi, toutes les fonctions $f_i(z), f_i(z^{q^l}), \ldots, f_i(z^{q^{l+n_i}})$ sont d\'efinies au point 
$\alpha$. En mettant bout \`a bout les syst\`emes compagnons associés à ces équations, 
on obtient un syst\`eme diagonal par bloc de la forme : 
$$
\left( \begin{array}{ c }
     1 \\
     f_1(z) \\
     \vdots \\
     f_{1}(z^{q^{l+n_1-1}})\\
     1 \\
     f_2(z)\\
     \vdots \\
     f_{r}(z^{q^{l+n_r-1}})
  \end{array} \right)
  = 
  \left(\begin{array}{ccc} B_1(z) & & \\ & \ddots & \\ & & B_r(z) \end{array} \right)
  \left( \begin{array}{ c }
     1 \\
     f_1(z^{q^l}) \\
     \vdots \\
     f_{1}(z^{q^{l+n_1}})\\
     1 \\
     f_2(z^{q^l})\\
     \vdots \\
     f_{r}(z^{q^{l+n_r}})
  \end{array} \right)\,.
$$
Le th\'eor\`eme 6.1 de \cite{AF} nous permet d'obtenir une base $\mathcal B$ 
de l'espace ${\rm Rel}_{\bf k(z)}(1,f_1(z),\cdots,f_1(z^{q^{l+n_1-1}}),\ldots,f_r(z^{q^{l+n_r-1}}))$. 
Soit $S$ la codimension de cet espace et $s$ la dimension de l'espace vectoriel engendré par les fonctions $f_1(z),\ldots, f_r(z)$. On choisit, parmi les fonctions $f_1(z),\ldots,f_r(z)$, des fonctions $g_1(z),\ldots,g_s(z)$ lin\'eairement ind\'ependantes telles que 
$$
\left( \begin{array}{ c }
     f_1(z) \\
     \vdots \\
     f_r(z)
  \end{array} \right) = \Gamma(z) \left( \begin{array}{ c }
     g_1(z) \\
     \vdots \\
     g_s(z)
  \end{array} \right) 
$$ 
où la matrice $\Gamma(z)$ est d\'efinie au point $\alpha$. Un tel choix de 
$\Gamma(z)$ se calcule explicitement, comme expliqué ci-après. 
On construit à partir de $\mathcal B$ une base $\mathcal B'$ de  
${\rm Rel}_{\bf k(z)}(f_1(z),\cdots,f_r(z))$ et on 
consid\`ere un premier vecteur ${\bf q}_1(z):=(q_{1,1}(z),\ldots,q_{1,r}(z))$ 
dans $\mathcal B'$ de sorte que 
\begin{equation}
\label{eq: premiererelation}
q_{1,1}(z)f_1(z)+\cdots+q_{1,r}(z)f_r(z) = 0\ ,
\end{equation}
o\`u les $q_{1,i}(z)$ sont des polyn\^omes premiers entre eux, et on choisit un indice $i_1$ pour lequel $q_{1,i_1}(\alpha) \neq 0$. On a 
$$
f_{i_1}(z) = \sum_{i \neq i_1} \frac{-q_{1,i}(z)}{q_{1,i_1}(z)} f_i(z)\, .
$$
On choisit alors un second vecteur ${\bf q}_2(z)$ dans ${\mathcal B}'$. On peut toujours supposer que $q_{2,i_1}(z)=0$, quitte à remplacer ${\bf q}_2(z)$ par une combinaison linéaire 
de ${\bf q}_2(z)$ et ${\bf q}_1(z)$. 
Il vient 
$$
\sum_{i\neq i_1} q_{2,i}(z)f_i(z) = 0\ .
$$
On fixe ensuite $i_2$, tel que $q_{2,i_2}(\alpha) \neq 0$ et on \'ecrit
$$
f_{i_2}(z) = \sum_{i \neq i_1,i_2} \frac{-q_{2,i}(z)}{q_{2,i_2}(z)} f_i(z)\ .
$$
En itérant ce procédé, on obtient des entiers distincts $i_1,i_2,\ldots,i_{r-s}$ tels que :
\begin{equation}\label{eq: ik}
f_{i_k}(z) = \sum_{i \neq i_1,i_2,\ldots,i_{k}} \frac{-q_{k,i}(z)}{q_{k,i_k}(z)} f_i(z)\ 
\end{equation}
et $q_{k,i_k}(\alpha)\not=0$ pour tout $k$, $1\leq k \leq r-s$. 
On choisit $g_1,\ldots,g_s$ de sorte que 
$\{g_1,\ldots,g_s\}=\{f_i : i \not= i_1,i_2,\ldots,i_{r-s}\}$. 
Les équations \eqref{eq: ik} permettent d'exprimer chaque $f_{i_{k}}$ comme 
une combinaison linéaire définie en $\alpha$ des fonctions $g_1,\ldots,g_s$, d'où l'on tire 
$\Gamma(z)$.  

On a alors l'inclusion suivante entre l'ensemble ${\rm Rel}_{{\Q}}(f_1(\alpha),\ldots,f_r(\alpha))$ et l'ensemble ${\rm Rel}_{{\Q}}(g_1(\alpha),\ldots,g_n(\alpha))$ : 
\begin{equation}
\label{eq:equivalence_relations}
{\rm Rel}_{{\Q}}(f_1(\alpha),\ldots,f_r(\alpha))\cdot \Gamma(\alpha) 
\subset {\rm Rel}_{{\Q}}(g_1(\alpha),\ldots,g_s(\alpha)) \, .
\end{equation}
En appliquant l'algorithme \ref{algo:indep}, on compl\`ete les fonctions $g_1(z),\ldots,g_s(z)$ en un système 
$$
\left( \begin{array}{ c }
     g_1(z) \\
     \vdots \\
     g_S(z)
  \end{array} \right) = A(z)\left( \begin{array}{ c }
     g_1(z^{q^l}) \\
     \vdots \\
     g_S(z^{q^l})
  \end{array} \right)  
$$
dans lequel les fonctions $g_1(z),\ldots,g_S(z)$ sont lin\'eairement ind\'ependantes. Comme les fonctions $g_1,\ldots,g_S$ sont définies en $\alpha$, le théorème 1.10 de \cite{AF} implique que   
la matrice $A(z)$ est bien d\'efinie en $\alpha$.  
D'apr\`es le th\'eor\`eme 1.9 de \cite{AF}, on a l'\'egalit\'e
$$
{\rm Rel}_{{\Q}}(g_1(\alpha),\ldots,g_S(\alpha))
= {\ker}_g A(\alpha),
$$
où ${\ker}_g A(\alpha)$ désigne le noyau à gauche de $A(\alpha)$. 
En ne consid\'erant que les vecteurs de ${\rm Rel}_{{\Q}}(g_1(\alpha),\ldots,g_S(\alpha))$ 
nuls sur les $S-s$ derni\`eres coordonn\'ees, on peut d\'eterminer une base 
$\bmu_1,\ldots,\bmu_t$ de l'espace ${\rm Rel}_{{\Q}}(g_1(\alpha),\ldots,g_s(\alpha))$. 
On peut alors calculer de mani\`ere explicite une base de l'ensemble ${\rm Rel}_{{\Q}}(f_1(\alpha),\ldots,f_r(\alpha))$, en r\'esolvant dans $\Q^r$ les syst\`emes linéaires 
$$
(x_1,\ldots,x_r).\Gamma(\alpha)=\bmu_i\ \,,
$$
 ce qui termine cette démonstration. 
\end{proof}

\section{Détermination des solutions analytiques d'une \'equation mahl\'erienne}
\label{par:solutions}

Dans cette partie nous reprenons les travaux de Dumas \cite{Dumas}. Nous montrons comment calculer de mani\`ere effective, \`a partir des premiers coefficients d'une solution d'\'equation mahl\'erienne, des coefficients arbitrairement \'elev\'es de la fonction. Les d\'emonstrations fourniront au passage une m\'ethode pour d\'eterminer une base de solutions analytiques d'une \'equation mahl\'erienne. Nous abordons s\'epar\'ement le cas d'une \'equation de type \eqref{eq: mahler} et d'un  système de type \eqref{eq: systeme}. 

\'Etant donnée une \'equation de la forme \eqref{eq: mahler}, notons 
$\nu_i$ la valuation du polyn\^ome $p_i(z)$ \`a l'origine pour $0\leq i\leq n$. 
D\'efinissons alors l'entier $d$ par :  
\begin{equation}
\label{eq: entiercritique}
d:= \max \left\{\ \left\lfloor \frac{\nu_0 - \nu_i}{q^i -1} \right\rfloor, \; 1 \leq i\leq  n \right\}\,.
\end{equation}
De \cite[Th\'eor\`eme 5]{Dumas}, on tire le r\'esultat suivant, que nous red\'emontrons ici.

\begin{lem}
\label{Dumas}
Soit $f(z):= \sum_{k=0}^\infty f_k z^k$ une solution de l'équation \eqref{eq: mahler}.  
Alors, si $d<0$, cette solution est unique et l'équation permet de déterminer tous les coefficents 
$f_k$. De plus $f(z)$ est nulle si, et seulement si, $p_{-1}(z)=0$. 
Si  $d\geq 0$, les coefficients $f_k$, pour $k> d$, 
sont d\'etermin\'es de mani\`ere unique par les coefficients $f_k$, pour $k\leq d$. 
\end{lem}

\begin{proof}
Si $k$ appartient à $\mathbb Q\setminus \mathbb N$, on pose $f_k := 0$. 
Dans l'\'equation \eqref{eq: mahler}, on isole $f(z)$
\begin{equation}
\label{equationremani\'ee}
f(z) = - \frac{p_{-1}(z)}{p_0(z)} - \frac{p_{1}(z)}{p_0(z)}f(z^q) - \cdots - \frac{p_{-n}(z)}{p_0(z)}f(z^{q^n}).
\end{equation} 
Consid\'erons alors le d\'eveloppement en s\'erie de Laurent des fractions rationnelles 
$-p_{i}(z)/p_0(z)$, $-1\leq i \leq n$, et notons : 
$$
-\frac{p_{i}(z)}{p_0(z)} := \sum_{k \geq \nu_i-\nu_0} p_{i,k} z^k \,.
$$
Pour tout entier $k\geq 0$, 
l'\'etude du terme de valuation $k$ dans l'\'equation \eqref{equationremani\'ee} donne la relation suivante :
\begin{equation}
\label{equationcoeff}
f_k = p_{-1,k} + \sum_{i=1}^n \sum_{j= \nu_i-\nu_0}^k p_{i,j}f_{\frac{k-j}{q^i}} \, .
\end{equation}
Si $k > d$, alors
$$ 
k > \frac{\nu_0 - \nu_i}{q^i -1}
$$
pour chaque $i$, $1\leq i\leq n$. On v\'erifie que cela implique, 
pour tout $j \geq \nu_i - \nu_0$ que 
$$ 
\frac{k-j}{q^i} < k \,.
$$
On en d\'eduit que le coefficient $f_k$ est d\'etermin\'e de mani\`ere unique par 
les coefficients $f_l$ pour $l<k$. 

En particulier, si $d<0$,  les coeffcients $f_k$ sont tous uniquement déterminés 
par \eqref{equationcoeff}. 
Cette dernière implique de plus que ces coefficients sont tous nuls si, 
et seulement si, $p_{-1,k}=0$ pour tout $k$. 
\end{proof}

\'Etant donn\'ee une \'equation de la forme \eqref{eq: mahler}, le lemme \ref{Dumas} nous permet 
de d\'eterminer l'espace affine des solutions analytiques. 
Si $d<0$, cet espace est réduit à une unique fonction, laquelle est entièrement déterminée par  \eqref{equationcoeff}. 
Si $d\geq 0$, il suffit de r\'esoudre 
dans $\k^{d+1}$ les $d+1$ \'equations affines \eqref{equationcoeff}, pour $0\leq k\leq  d$. 
On déduit immédiatement le r\'esultat suivant.

\begin{coro}
\label{dimensionepsacesolutions}
\'Etant donn\'e une \'equation de type \eqref{eq: mahler}, notons $\nu_0$ la 
valuation du polyn\^ome $p_0(z)$. La dimension de l'espace affine des solutions analytiques  
de cette équation est
au plus $1+ \frac{\nu_0}{q-1}$. 
En particulier, on peut la majorer 
ind\'ependamment de $n$.
\end{coro}

Pour les systèmes de type \eqref{eq: systeme}, nous obtenons un r\'esultat similaire.

\begin{lem}
\label{conditions initiales}
Soit $\f(z) = \sum_{k=0}^\infty \f_k z^k$ un vecteur non nul de fonctions analytiques 
solution d'un système de type \eqref{eq: systeme}.  
Soit $\nu$ le minimum des valuations en $z=0$ des coefficients de la matrice $A(z)$. 
Alors, on a $\nu \leq 0$ et les vecteurs de coefficients  $\f_k$ pour $k>\left\lfloor \frac{-\nu}{q-1} \right\rfloor$ sont 
d\'etermin\'es de mani\`ere unique par les vecteurs de coefficients 
$\f_k$ pour $k\leq \left\lfloor \frac{-\nu}{q-1} \right\rfloor$. 
\end{lem}

\begin{proof}
Les coefficients de la matrice $A(z)$ sont des fractions rationnelles. 
On peut donc consid\'erer le d\'eveloppement en s\'erie de Laurent de la matrice $A(z)$,
$$
A(z) := \sum_{k\geq \nu} A_k z^k, \qquad A_k \in \mathcal{M}_n(\k)
$$
o\`u $A_\nu$ est une matrice non nulle. 
La relation matricielle entre $\f(z)$ et $\f(z^q)$ implique les relations suivantes 
entre les coefficients 
\begin{equation}
\label{eq: recurrence}
\f_k = \sum_{r=\nu}^k A_r \f_{\frac{k-r}{q}} \,,
\end{equation}
o\`u l'on convient que $\f_{\frac{k-r}{q}}=0$ si $q$ ne divise pas $k-r$. 
Si $\nu >0$, alors $\f$ est identiquement nulle. En effet, en considérant $k_0$ le plus petit entier tel 
que $\f_{k_0}\not=0$, l'équation \eqref{eq: recurrence} serait contradictoire.  
Supposons maintenant que $\nu \leq 0$.   
Si l'entier $k$ est strictement sup\'erieur \`a $\frac{-\nu}{q-1}$, 
alors $kq > k - \nu$ et pour tout $r \geq \nu$, 
$$
k > \frac{k-r}{q}
$$
et l'équation \eqref{eq: recurrence} permet de définir de manière unique le vecteur $\f_k$ 
en fonctions des vecteurs $\f_l$ avec $0\leq l<k$.
\end{proof}



\begin{thebibliography}{99}

\bibitem{AB} B. Adamczewski and J. Bell, {\it A problem around Mahler functions},  
Ann. Sc. Norm. Super. Pisa, to appear. 

\bibitem{AF} B. Adamczewski, C. Faverjon, {\it M\'ethode de Mahler : relations lin\'eaires, transcendance et applications aux nombres automatiques }, pr\'etirage 2015, {\tt arXiv:1508.07158 [math.NT]}.

\bibitem{Dumas} P. Dumas, {\it R\'ecurrences mahl\'eriennes, suites automatiques, \'etudes asymptotique
Math\'ematiques}, Thèse, Université de Bordeaux I, Talence, 1993. 

\bibitem{PPH} P. Philippon, {\it Groupes de Galois et nombres automatiques}, J. Lond. Math. Soc. 
{\bf 92} (2015),  596--614.

\bibitem{Ra} B. Randé, {\it Equations fonctionnelles de Mahler et applications aux
suites $p$-régulières}, Thèse, Université de Bordeaux I, Talence, 1992.  
\end{thebibliography}
\end{document}